\documentclass[11pt]{article}
\usepackage{amsthm}
\usepackage{amsmath}
\usepackage{graphicx}

\newtheorem{theorem}{Theorem}
\newtheorem{remark}{Remark}
\newtheorem{proposition}{Proposition}
\newtheorem{lemma}{Lemma}

\newtheorem{example}{Example}

\renewcommand{\int}{{\rm int\,}}
\newcommand{\transp}{{^{\rm T}}}
\newcommand{\ri}{{\rm ri\,}}
\newcommand{\Lin}{{\rm Lin\,}}
\newcommand{\co}{\mathop{\rm co\,}}
\newcommand{\cone}{{\rm cone\,}}
\newcommand{\B}{{\rm I\!B}}
\newcommand{\R}{{\rm I\!R}}
\renewcommand{\L}{\mathcal{L}}
\newcommand{\matr}[1]{\begin{bmatrix} #1 \end{bmatrix}}    
\def\transp{^{\rm T}}

\author{Javier Pe\~na\thanks{Tepper School of Business,
Carnegie Mellon University, USA, {\tt jfp@andrew.cmu.edu}} \and Vera Roshchina\thanks{CIMA, Universidade de \'Evora, Portugal, {\tt vera.roshchina@gmail.com}}}
\title{A Complementarity Partition Theorem for Multifold Conic Systems}

\begin{document}
\maketitle

\begin{abstract} Consider a homogeneous multifold convex conic system
\[
Ax = 0, \; x\in K_1\times \cdots \times K_r
\]
and its alternative system
\[
A\transp y \in K_1^*\times \cdots \times K_r^*,
\]
where $K_1,\dots, K_r$ are regular closed convex cones.
We show that there is canonical partition of the index set $\{1,\dots,r\}$ determined by certain complementarity sets associated to the most interior solutions to the two systems.  Our results are inspired by and extend the Goldman-Tucker Theorem for linear programming.

\end{abstract}


\section{Introduction}

Assume $K\subseteq \R^n$ is a closed convex cone and $A\in \R^{m\times n}$.  Consider the homogeneous conic system
\begin{equation}\label{eq:P}
Ax = 0, \quad x\in K,\tag{P}
\end{equation}
and its alternative system
\begin{equation}\label{eq:D}
 A\transp y \in K^*,\tag{D}
\end{equation}
where  $K^*\in \R^n$ is the dual of $K^*$.  It is immediate that any solutions $\bar x$ and $\bar y$ to \eqref{eq:P} and \eqref{eq:D} respectively are {\em complementary,} that is, they satisfy $\bar y\transp (A\bar x) = 0$.  In particular, if either \eqref{eq:P} or \eqref{eq:D} has a strict feasible solution, then the other one only has trivial solutions.  In the special case when $K=\R^n_+$, a stronger related property holds.  As a consequence of Goldman-Tucker Theorem~\cite{GoldmanTucker}, there always exist solutions $\bar x$ and $\bar y$ to \eqref{eq:P} and \eqref{eq:D} respectively such that $\bar x + A\transp \bar y \in \R^n_{++}$.  Such pairs of  {\em strictly complementary} solutions are associated to a canonical partition $B\cup N = \{1,\dots,n\}$ of the index set $\{1,\dots,n\}$ (see Proposition~\ref{thm:LP_compl} below).  The partition sets $B$ and $N$ correspond to the most interior solutions to \eqref{eq:P} and \eqref{eq:D} respectively. Furthermore, there is a nice geometric interpretation of the sets $B,N$ (see Proposition~\ref{prop:three.ways} below).

\medskip

We present a generalization of the above strict complementary results to more general conic systems.  To that end, we consider the case when the cone $K$ is the direct product of $r$
lower-dimensional regular closed convex cones. That is, we assume
\begin{equation}\label{eq:multifold.cone}
K = K_1\times\cdots
\times K_r,
\end{equation}
where $K_i\subseteq \R^{n_i}$ is a regular closed convex cone for $i=1,\dots,r.$ Throughout the sequel we shall use $I$ to denote the set $I= \{1,\dots, r\}$ and
$n$ to denote the dimension $n = \sum_{i=1}^r n_i$.

Following the terminology introduced in~\cite{CCP_Multifold2008} we call the conic systems \eqref{eq:P} and \eqref{eq:D} {\em
multifold} when the cone $K$ is as in \eqref{eq:multifold.cone}.
This type of multifold structure is common in optimization. Formulations for linear programming (LP), second-order conic programming (SOCP) and semidefinite programming (SDP) problems generally lead to feasibility problems of this form.  Our first main result  (Theorem~\ref{thm:main_geometry}) shows that there are some canonical subsets $B,N$ and $B_0,N_0$ of $I$ associated to certain geometric properties of the problems \eqref{eq:P} and \eqref{eq:D}.  These sets generalize the partition sets $B,N$ in the case $K=\R^n_+$.  Our second main result (Theorem~\ref{thm:main_partition}) shows that there exists a unique canonical partition of the index set $I$ associated to the most interior solutions to \eqref{eq:P} and \eqref{eq:D}.

The paper is organized as follows.  Section~\ref{sec:preamble} provides the foundation for our work, namely the existence of strictly complementary solutions to \eqref{eq:P}, \eqref{eq:D} when $K=\R^n_+$.  Section~\ref{sec:main_section} presents our main results, namely Theorem~\ref{thm:main_geometry} and
Theorem~\ref{thm:main_partition}.  Section~\ref{sec:second_order} discusses in more detail the special case of second-order conic systems.  Section~\ref{sec:final} concludes the paper with some final remarks.

\section{Strict Partition for Polyhedral Homogeneous Systems}
\label{sec:preamble}
To motivate and state our main results, we first consider the special case when $K = \R^n_+$ in \eqref{eq:P}, \eqref{eq:D}.  In this case the conic systems become
\begin{equation}\label{eq:P.LP}
Ax = 0, \quad x\ge 0;
\end{equation}
and
\begin{equation}\label{eq:D.LP}
 A\transp y \ge 0,
\end{equation}
where $A\in \R^{m\times n}$. This can be considered as a special case of a multifold conic system with $r=n$ and $K_i = \R_+$ in \eqref{eq:multifold.cone}.  Hence throughout this section we have $I = \{1,\dots,n\}$.  Furthermore, for notational convenience, we shall write $A = \matr{a_1 & \cdots & a_n} \in \R^{m\times n}$.  In other words, $a_i \in\R^n$ is the $i$-th column of $A$. The following result is a consequence of the Goldman-Tucker Theorem for linear programming~\cite{GoldmanTucker}.

\begin{proposition}\label{thm:LP_compl} Consider the pair of feasibility problems \eqref{eq:P.LP} and \eqref{eq:D.LP} for a given $A\in\R^{m\times n}$. For a unique partition $B\cup N = I$ of the index set $I$ there exist solutions $\bar x$ to \eqref{eq:P.LP} and $\bar y$ to \eqref{eq:D.LP} satisfying
\[
\bar x_B > 0, \; \; A_N\transp \bar y > 0,
\]
where we have used the standard notation: $\bar x_B>0$ means $\bar x_i>0$ for all $i\in B$, and $A_N\transp \bar y>0$ means $a_i\transp \bar y>0$ for all $i\in N$.
\end{proposition}

The partition sets $B,N$ in Proposition~\ref{thm:LP_compl} can be described is several ways.  The three descriptions of $B,N$ displayed  in Proposition~\ref{prop:three.ways} below lay the foundation for our main work.
In the sequel we use the following convenient notation.  For a convex cone $C\subseteq  \R^d,$ let $\Lin(C) \subseteq \R^d$  denote the {\em lineality space} of $C$, that is, the largest linear subspace contained in $C$.  Observe that because $C$ is a convex cone,  $\Lin(C) = \{x\,|\,x,-x\in C\}.$

\begin{proposition}\label{prop:three.ways} The sets $B,N$ in Proposition~\ref{thm:LP_compl} can be described as
\begin{equation}\label{eq:first}
\begin{array}{c}
B = \{i\in I\,|\, \exists\, x: \, Ax = 0,\, x\geq 0, \, x_i>0\},
\\ [2ex] 
N = \{i\in I\,|\, \exists\, y: \, A\transp y \geq 0,\, a_i\transp y >0\}.
\end{array}
\end{equation}
These sets can also be described as
\begin{equation}\label{eq:second}
\begin{array}{c}B = \{i\in I\,|\,  A\transp y \geq 0\Rightarrow a_i\transp y=0\},\\[2ex]
N = \{i\in I\,|\,  \, Ax = 0,\, x\geq 0 \Rightarrow x_i=0\}.
\end{array}
\end{equation}
And they can also be described as
\begin{equation}\label{eq:third}
\begin{array}{c}
B = \{i\in I\,|\,  a_i\in \Lin(A\R^n_+)\},\\[2ex]
N = \{i\in I\,|\,  a_i\not \in \Lin(A\R^n_+)\}.
\end{array}
\end{equation}
\end{proposition}

The description \eqref{eq:third} of the sets $B,N$ has an interest{color{red}{ing}} geometric interpretation.  What determines if a particular index $i$ belongs to $B$ or $N$ is whether the corresponding $i$-th column $a_i$ lies in the lineality space of the cone $A\R^n_+$.  This geometric interpretation has an interesting extension to multifold conic systems as Theorem~\ref{thm:main_geometry} below shows.  Proposition~\ref{prop:three.ways} is a consequence of Farkas Lemma and is also a special case of Theorem~\ref{thm:main_geometry} below.

\section{A Canonical Partition for Multifold Conic Systems}
\label{sec:main_section}
Consider now the general conic systems \eqref{eq:P}, \eqref{eq:D} where $A\in \R^{m\times n}$ and the cone $K\subseteq \R^n$ is as in \eqref{eq:multifold.cone}.
For notational convenience, write $A = \matr{A_1 & \cdots & A_r}$, where $A_i \in \R^{m\times n_i}$ is the $i$-th block of the matrix $A$.

Our main results generalize Proposition~\ref{thm:LP_compl}  and Proposition~\ref{prop:three.ways} to multifold conic systems.
Motivated by \eqref{eq:first}, define
\begin{equation}\label{eq:first.conic}
\begin{array}{cc}
B = \{i\in I\,|\, \exists\, x: \, Ax = 0, \, x\in K,\, x_i\in \int K_i\},
\\[2ex]
N = \{i\in I\,|\, \exists\, y: \, A\transp y \in K^*,\, A_i\transp y \in \int K^*_i\}.
\end{array}
\end{equation}
Likewise, motivated by \eqref{eq:second}, define
\begin{equation}\label{eq:second.conic}
\begin{array}{cc}
B_0 = \{i\in I\,|\,  A\transp y \in K^* \Rightarrow A_i\transp y = 0\},
\\[2ex]
N_0 = \{i\in I\,|\,  Ax=0,\,  x\in K \Rightarrow x_i = 0\}.
\end{array}
\end{equation}

We are now ready to state our main results.  The following theorem
establishes a characterization of the index sets $B,B_0,N,N_0$ in terms of the geometry of the sets $AK$ and $A_iK_i$.  In the statement below, $\overline{AK}$ denotes the closure of $AK$.

\begin{theorem}\label{thm:main_geometry}
\begin{description}
\item[(i)]
The sets $B$, $N$ defined in \eqref{eq:first.conic} can also be described as
\begin{equation}\label{eq:third.conic}
\begin{array}{c}
B = \{i\in I\,|\,  \ri A_i K_i\cap \Lin(AK) \neq \emptyset \},\\[2ex]
N = \{i\in I\,|\,  A_i (K_i\setminus\{ 0\}) \cap \Lin(\overline{AK}) = \emptyset\}.
\end{array}
\end{equation}
\item[(ii)]
The sets $B_0$, $N_0$ defined in \eqref{eq:second.conic} can also be described as
\begin{equation}\label{eq:third.conic.0}
\begin{array}{c}
B_0 = \{i\in I\,|\,  \ri A_i K_i\cap \Lin(\overline{AK}) \neq \emptyset \},\\[2ex]
N_0 = \{i\in I\,|\,  A_i (K_i\setminus\{ 0\}) \cap \Lin(AK) = \emptyset\}.
\end{array}
\end{equation}
\end{description}
\end{theorem}
To ease exposition, we defer the proof of Theorem~\ref{thm:main_geometry} to the end of this Section.

\medskip

Observe that in the case when $K$ is a polyhedral cone, we have $AK = \overline{AK}$. Thus for $K$ polyhedral Theorem~\ref{thm:main_geometry} yields $B=B_0$ and $N = N_0$. In particular Proposition~\ref{prop:three.ways} readily follows from Theorem~\ref{thm:main_geometry}.

The next theorem generalizes Proposition~\ref{thm:LP_compl}.  It shows that there is a unique canonical partition of the index set $I$ into six complementarity subsets of indices.

\begin{theorem}\label{thm:main_partition} For a unique partition $B \cup B' \cup N \cup N' \cup C \cup O = I$ of the index set $I$ the following three properties hold:
\begin{description}
  \item[(i)] There exists a solution $\bar x$ to
      \eqref{eq:P}
      such that $$\bar x_i\in \int K \text{ for all } i\in B \text{ and } x_i
      \neq 0 \text{ for all } i\in B'\cup C, $$
  \item[(ii)] There exists a solution $\bar y$ to  \eqref{eq:D}     such that
      $$A_i\transp \bar y\in \int
      K_i^* \text{ for all } i\in N, \text{ and } A_i\transp \bar y\neq 0 \text{ for
      all } i\in N'\cup C.$$
  \item[(iii)] For any solutions $x$ to
      \eqref{eq:P} and  $y$ to
      \eqref{eq:D} we have
      \[x_i =0 \text{ for all } i\in
     N'\cup N\cup O \text{ and } A_i\transp y=0 \text{ for all } i\in
      B\cup B'\cup O.\]
\end{description}

\end{theorem}
\begin{proof}
Take $B,N$ and $B_0,N_0$ as in \eqref{eq:first.conic} and \eqref{eq:second.conic} respectively, and let
\begin{equation}\label{eq:partition} B' := B_0\setminus (B\cup N_0); \;\; N' = N_0\setminus (N\cup B_0); \;\;O = B_0\cap N_0; \;\; C = I\setminus(B_0\cup N_0).\end{equation}
The sets $B,B',C,N,N',O$ comprise a partition of $I$ because  by  Theorem~\ref{thm:main_geometry}
$B \subseteq B_0, \; N \subseteq N_0,$ and also $B \cap N_0 = N\cap B_0 = \emptyset$.

We next prove part (i).  By Theorem~\ref{thm:main_geometry}(ii), for every $i\notin N_0 $ there exists a solution $x^{(i)}$ to \eqref{eq:P} such that $ x^{(i)}_i\in K_i\setminus \{0
\}$. Hence $x_{N_0}=\sum_{i\in I\setminus N_0} x^{(i)}$ is
solution to \eqref{eq:P} and for every $i\notin N_0$ we have $x_i \neq 0$ (since $K_i$ is pointed). By
the definition of $B$, for each $i\in B$ there exists a solution
$\bar x^{(i)}$  to \eqref{eq:P} such that $ x^{(i)}_i\in \int K_i$.
Then $x_{B}=\sum_{i\in B} x^{(i)}$ is solution to \eqref{eq:P} and  $ (x_B)_i\in \int K_i$ (by \cite[Lemma A.2.1.6]{JBHULem2001}). Therefore, again by the pointedness of each $K_i$ and by \cite[Lemma A.2.1.6]{JBHULem2001}, the point
$\bar x=x_B+x_{N_0}$ is a solution to \eqref{eq:P} such that
$\bar x_i\in \int K$ for all $i\in B$ and $x_i
      \neq 0$ for all $i\in B'\cup C.$
An analogous argument proves part (ii). Part (iii) follows directly from the  definition \eqref{eq:second.conic} of $B_0 = B \cup B' \cup O$ and $N_0 = N \cup N' \cup O$.

The uniqueness of the partition is proven as follows. First, observe that if (i) and (ii) hold, then by construction $B,N$ must be as in \eqref{eq:first.conic}.  Likewise if (iii) holds, then $B_0,N_0$ must be as in \eqref{eq:second.conic}.   Therefore if (i), (ii), and (iii) hold, the sets $B,B',C,N,N',O$ must be as in \eqref{eq:partition}.
\end{proof}
The Venn diagram representing the relations between the subsets
of $B$, $B'$, $N$, $N'$, $C$ and $O$ of $I$ is given in
Fig.~\ref{fig:diagram}.
\begin{figure}[h]
\centering
\includegraphics[keepaspectratio, height=120pt ]{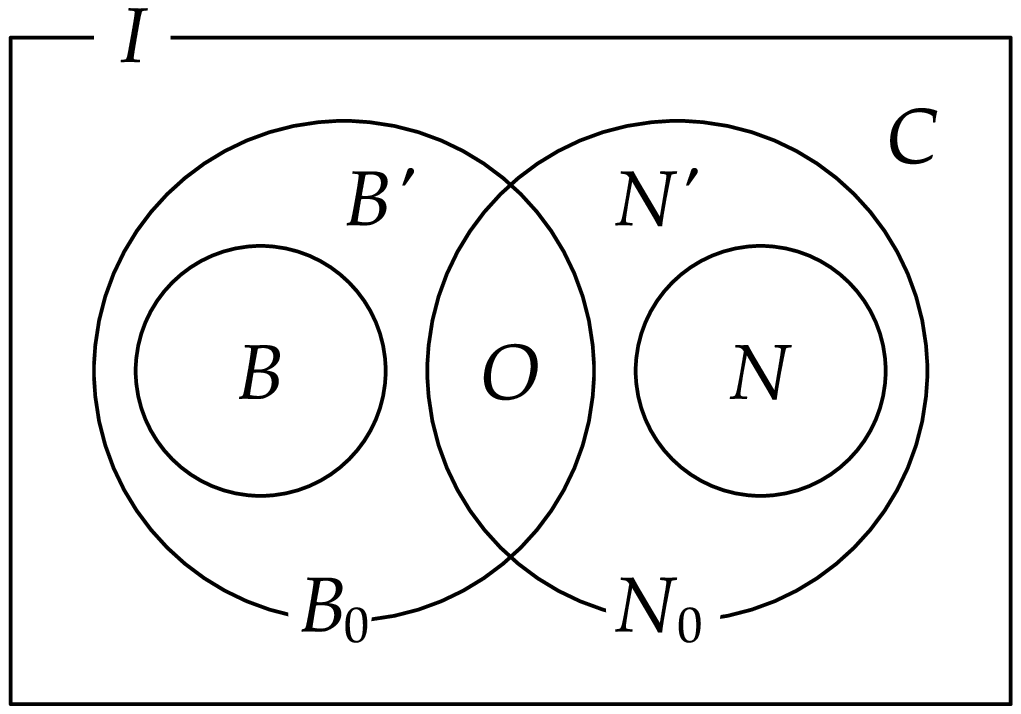}
\caption{Partition of $I$ into six disjoint sets based on $B$, $N$, $B_0$ and $N_0$}
\label{fig:diagram}
\end{figure}
In Section~\ref{sec:geom} we provide an example of a
second-order conic programming problem for which all six sets
are nonempty.  It should be noted that a six-set partition for second-order conic programs similar to the one suggested here was mentioned in
\cite[Section~6]{AlizGold03}.  However, there was no prior characterization of this partition along the lines of Theorem~\ref{thm:main_geometry}.

\medskip

We conclude this section with the proof of Theorem~\ref{thm:main_geometry}.  Our proof relies on the following separation lemma.  Although this result is likely known, we were not able to locate it in the literature in this exact form.

\begin{lemma}\label{lem:sep_cones} Let $K_1,K_2\subseteq \R^n$ be
closed convex cones such that $K_1\cap K_2 = \{0\}$ and
$\Lin(K_2) = \{0\}$.  Then $K_1$ and $K_2$ can be strictly
separated in the following sense. There exists $s\in \R^n$ such that
\begin{equation}\label{eq:005}
\langle s, y\rangle\leq  0\quad  \forall y\in K_1,\qquad \langle s, y\rangle>0 \quad \forall y\in K_2\setminus \{0\}.
\end{equation}
\end{lemma}
\begin{proof} Let $C:=\{x\in K_2\,|\, \|x\|=1\}$. Since $K_2$ is closed and $\Lin K_2 = \{0\}$, the set  $\co C$ is compact and $0\notin \co C.$ In particular $K_1 \cap \co C = \emptyset$.
Hence, by~\cite[Corol. A.4.1.3]{JBHULem2001}, there exists a point $s\in \R^n$ such that
\begin{equation}\label{eq:007}
\sup_{y\in K_1}\langle s, y\rangle <\min_{y\in {\co C}}\langle s, y\rangle.
\end{equation}
Since $0\in K_1$ we have
$
\sup_{y\in K_1}\langle s, y\rangle \geq \langle s, 0\rangle =0.
$
Thus from \eqref{eq:007} and the fact that $K_1$ is a cone it follows that
\[
\sup_{y\in K_1}\langle s, y\rangle = 0 < \min_{y\in {C}}\langle s, y\rangle,
\]
and \eqref{eq:005} readily follows.
\end{proof}

\noindent
{\em Proof of Theorem~\ref{thm:main_geometry}.}
\begin{description}
\item[(i)] We first show $B \supseteq \{i\in I\,|\,  \ri A_i K_i\cap \Lin(AK) \neq \emptyset \}$.  Assume $i\in I$ is such that $\ri (A_i K_i)\cap \Lin(AK) \neq \emptyset$. By~\cite[Prop. A.2.1.12]{JBHULem2001}, $\ri (A_i K_i) = A_i(\ri
K_i)=A_i(\int K_i)$. Hence there exists $\bar x_i\in \int K_i$
such that $A_ix_i\in \Lin(AK)$.  Thus $-A_ix_i = Ax'$ for some $x'\in K$. Let $x\in K$ be defined by putting $x_j= x'_j$ for $j\neq i$ and $x_i = x'_i+\bar x_i$. By~\cite[Lemma A.2.1.6]{JBHULem2001}, it follows that $x$ is a solution to \eqref{eq:P} and $x_i \in \int K_i$.  Thus $i\in B$.

Next, we show $B \subseteq \{i\in I\,|\,  \ri A_i K_i\cap \Lin(AK) \neq \emptyset \}$.
Assume $i\in B$. Hence there exists $x\in K$ such that $x_i \in \int
K_i=\ri K_i$ and $Ax = 0$. By \cite[Prop. A.2.1.12]{JBHULem2001},
\begin{equation}\label{eq:001}
A_i x_i \in \ri A_i K_i.
\end{equation}
Let $x'\in\R^n$ be defined by putting $x'_j = 0 $ for $j\neq i$ and $x'_i = x_i$.
We have $\bar x = x-x'\in K$ and so $-A_i x_i = A\bar x \in AK$.
But $A_ix_i\in A_i K_i \subset AK$ as well, therefore
\begin{equation}\label{eq:002}
A_ix_i \in \Lin(AK).
\end{equation}
From \eqref{eq:001} and \eqref{eq:002} we have $\ri A_i K_i\cap
\Lin(AK) \neq \emptyset$.

\medskip

Now we show $N \supseteq \{i\in I\,|\,  A_i (K_i\setminus\{ 0\}) \cap \Lin(\overline{AK}) = \emptyset\}.$  Assume $i\in I$ is such that $A_i (K_i\setminus\{ 0\}) \cap \Lin(\overline{AK}) = \emptyset$.
Since $A_i K_i\subseteq \overline{AK}$, this yields
$$
\Lin(A_iK_i) = \{0\} \; \text{ and } \;- A_i (K_i\setminus\{0\}) \cap \overline{AK} = \emptyset.
$$
Therefore by Lemma~\ref{lem:sep_cones} applied to $K_1 = \overline{AK}$ and $K_2 = -A_iK_i$,
 there exists a nonzero $y\in \R^m$ such that
$
y\transp Ax \geq 0 \quad \forall x\in K
$
and $
y\transp (-A_ix_i) < 0  \quad \forall x_i \in K_i\setminus\{0\}.
$ In particular, $y$ is a solution to \eqref{eq:D} and $A_i\transp y \in \int K_i^*.$

Next we show $N \subseteq \{i\in I\,|\,  A_i (K_i\setminus\{ 0\}) \cap \Lin(\overline{AK}) = \emptyset\}.$  To that end, we show the contrapositive.  Assume $i\in I$ is such that $A_i (K_i\setminus\{ 0\}) \cap \Lin(\overline{AK}) \ne \emptyset$.
Then there exists $x_i\in K_i\setminus \{0\}$ such that $A_i x_i, -A_i x_i
\in \overline{AK}$. Hence for any solution $y$ to \eqref{eq:D} we have
$
y\transp A_i x_i \geq 0$ and $\quad y\transp (-A_i x_i)\geq 0
$
so $y\transp A_i x_i  = 0$.  Since $x_i\in K_i\setminus \{0\}$, this implies that $A_i\transp y \not \in \int K_i^*$.  Consequently $i\not \in N$.

\item[(ii)] We first show $B_0 \supseteq \{i\in I:\ri A_i K_i\cap \Lin(\overline{AK})  \neq \emptyset  \}$. Assume that $\ri A_i K_i\cap \Lin(\overline{AK})  \neq
\emptyset $. Then by \cite[Prop. A.2.1.12]{JBHULem2001} there exists
$x_i\in \int K_i$ such that  $A_i x_i,-A_ix_i\in \overline{AK}$.
Therefore, as in the previous paragraph, it follows that
$y\transp A_i x_i = 0$
for any solution $y$ to \eqref{eq:D}. Since $x_i\in \int K_i$, this implies that
$A_i\transp y = 0$ for any solution $y$ to \eqref{eq:D}. Thus $i \in B_0.$

We next show $B_0 \subseteq \{i\in I:\ri A_i K_i\cap \Lin(\overline{AK})  \neq \emptyset  \}$.
Assume $i\in B_0$.  Then for all solutions $y$ to~\eqref{eq:D} and all $x_i\in K_i$ we have $(A_ix_i)\transp
y=0$. Thus  $A_ix_i, -A_ix_i \in \overline{AK}$ for all $x_i\in K_i$, and hence $A_iK_i\subset \Lin(\overline{AK})$.

\medskip

We now show  $N_0 \supseteq \{i\in I: A_i (K_i\setminus \{0\}) \cap \Lin(AK) = \emptyset \}.$  To that end, we show the contrapositive.  Assume  $i\in I$ is such that
there exists a solution $x$ to \eqref{eq:P} with $x_i \neq 0$.  Since $-A_ix_i=\sum_{j\neq i}A_jx_j \in AK$, we have $A_ix_i,-A_ix_i\in AK$ with $x_i\neq 0$. Hence $A_ix_i \in A_i (K_i\setminus \{0\}) \cap \Lin(AK).$

We finally show $N_0 \subseteq \{i\in I: A_i (K_i\setminus \{0\}) \cap \Lin(AK) = \emptyset \}.$
Again we show the contrapositive.  Assume $i\in I$ is such that $A_i (K_i\setminus \{0\}) \cap
\Lin(AK)  \neq \emptyset$. Then there exists $x_i\in K_i\setminus\{0\}$ such
that $A_ix_i,-A_ix_i\in AK$. In particular, for some
$x'\in K$ we have $-A_ix_i = Ax'.$ Then the point $\bar x\in K$ defined by putting
$\bar x_j = x'_j$ for $j\neq i$ and $\bar x_i = x'_i+x_i$ is a solution to \eqref{eq:P} with $\bar x_i \neq 0$ (because $K_i$ is pointed).

\end{description}
\qed

\section{Second-Order Conic Systems}
\label{sec:second_order}
Consider the special case when the cone $K$ in \eqref{eq:P},\eqref{eq:D} is a cartesian product of Lorentz cones.  In other words,
\begin{equation}\label{eq:second.order.cone}
K = \L_{n_1-1}\times \cdots \times \L_{n_r-1},
\end{equation}
where
$$
\L_{n_i-1} = \{(x_0,\bar x)\in \R^{n_i}\, |\, x_0\geq \|\bar x\|\}, \; i=1,\dots,r.
$$
Here  $\|\cdot\|$ is the Euclidean norm in $\R^{n_i}.$  We shall put, by convention, $\L_0 = \R_+$ when $n_i=1$.  Also, for $d \ge 1$ we will let $\B_{d} \subseteq \R^d$ denote the Euclidean closed unit ball in $\R^d$ centered at zero.

For each $i\in I$ assume the $i$-th block  $A_i\in \R^{m\times n_i}$ of $A$ is of the form
\[
A = \matr{A_{i0} & \bar A_i}, \; A_{i0} \in \R^m, \; \bar A_i \in \R^{m\times (n_i-1)}.
\]
In other words, $A_{i0} $ denotes the first
column of $A_i$, and $\bar A_i$ denotes the block of
remaining $n_i-1$ columns.   Put
\begin{equation}\label{eq:Ellipsoids}
E_i =
  \left\{
    \begin{array}{lll}
      A_{i0}+\bar A_i \B_{n_i-1},& \text{ if} & n_i>1,\\
      A_{i0}    ,& \text{ if} & n_i = 1.
    \end{array}
    \right.
\end{equation}
Observe that $AK = \cone \co_{i\in I}\{E_i\}.$
Theorem~\ref{thm:main_geometry} can now be stated in a way that more closely resembles \eqref{eq:third} in Proposition~\ref{prop:three.ways}.

\begin{proposition}\label{cor:main_socp} Consider the pair of multifold conic systems \eqref{eq:P}, \eqref{eq:D}.  Assume $K$ is as in \eqref{eq:second.order.cone} and $E_i,\;i\in I$ are as in \eqref{eq:Ellipsoids}.  Then
\begin{description}
\item[(i)]
The sets $B$, $N$ defined in \eqref{eq:first.conic} satisfy
\[
\begin{array}{c}
B = \{i\in I\,|\,  \ri E_i\cap \Lin(AK) \neq \emptyset \},\\[2ex]
N = \{i\in I\,|\,  E_i \cap \Lin(\overline{AK}) = \emptyset\}.
\end{array}
\]
\item[(ii)]
The sets $B_0$, $N_0$ defined in \eqref{eq:second.conic} satisfy
\[
\begin{array}{c}
B_0 = \{i\in I\,|\,  \ri E_i\cap \Lin(\overline{AK}) \neq \emptyset \},\\[2ex]
N_0 = \{i\in I\,|\,  E_i \cap \Lin(AK) = \emptyset\}.
\end{array}
\]
\end{description}
\end{proposition}
\begin{proof} This readily follows from Theorem~\ref{thm:main_geometry} and the construction of the sets $E_i,\; i\in I$.
\end{proof}

We now discuss an example of a second-order feasibility system where
all six sets $B,N,B',N',C,O$ in the partition of Theorem~\ref{thm:main_partition} are nonempty.

\begin{example}
{\em
Let $K = \R_+\times \R_+\times \R_+\times {\cal L}_1\times {\cal L}_1\times {\cal L}_3\subseteq
\R^{11}$ and
$$
A = \left[
  \begin{array}{rrrrrrrrrrrrrrrr}
   1  & &  0 & &  0 & &  1 & 0 & &   1 & 1 & &  0 & 1 & 0 & 0 \\
   0  & &  1 & &  0 & &  0 & 0 & &  -1 & 0 & &  0 & 0 & 1 & 0 \\
   0  & &  0 & &  1 & &  1 & 1 & &   0 & 0 & &  1 & 0 & 0 & 1
  \end{array}
\right].
$$
In this case,
$$
E_1 = \{(1,0,0)\},\;E_2 = \{(0,1,0)\},\;E_3 = \{(0,0,1)\},\; E_4 = \co\{(1,0,0),(1,0,2)\},
$$
$$
 E_5 = \co\{(0,-1,0),(2,-1,0)\},\; E_6 = \{(0,0,1)\}+\B_3.
$$
Thus
$$
AK = \{(x,y,z)\in \R^3\, |\, z>0\}\cup\{(x,y,z)\in \R^3\, |\, z=0, x\geq 0\},$$
$$\overline{AK} = \{(x,y,z)\in \R^3\, |\, z\geq 0\},
$$
and
$$
\Lin(AK) = \{0\}\times \R\times \{0\}; \quad \Lin(\overline{AK}) = \R\times\R\times \{0\}.
$$
Figure~\ref{fig:Example2} shows the sets $\Lin(AK), \Lin(\overline{AK}),E_1,\dots, E_6$.
\begin{figure}[ht]
\centering
\includegraphics[keepaspectratio, height=180pt ]{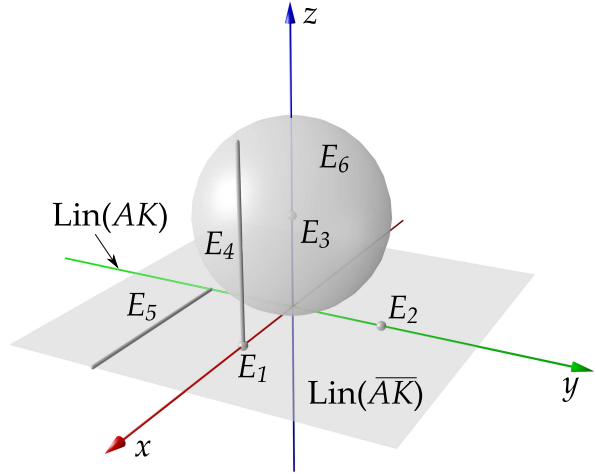}
\caption{Geometric interpretation of the partition in Example 1}
\label{fig:Example2}
\end{figure}

\noindent
From Proposition~\ref{cor:main_socp} we readily get
$$
B = \{2\},\; N = \{3\}, \; B_0 = \{1,2,5\}, \; N_0 = \{1,3,4\}.
$$
Hence in this case the partition sets of Theorem~\ref{thm:main_partition} are $$
O = \{1\},\; B = \{2\},\; N=\{3\},\; N'=\{4\},\; B'=\{5\},\; C = \{6\}.
$$

We note that in this small example the systems $Ax=0$, $x\in K$ and
$A\transp y \in K$ can be solved directly.  We obtain the following parametric families of solutions to \eqref{eq:P} and \eqref{eq:D} respectively:
$$
x = (0, \lambda, 0, 0 , 0, \lambda, -\lambda, \mu, 0, 0, -\mu), \quad \lambda\geq 0, \mu\geq 0;
$$
and
$$
y = \left(0,0,\gamma\right),
\quad \gamma\geq 0.
$$
The correctness of the partition $
O = \{1\},\; B = \{2\},\; N=\{3\},\; N'=\{4\},\; B'=\{5\},\; C = \{6\}
$
can then be directly verified.
}
\end{example}

\section{Some Final Remarks}
\label{sec:final}
\subsection{Geometric interpretation of
Theorem~\ref{thm:main_geometry}}\label{sec:geom}

Proposition~\ref{cor:main_socp} can be stated in a form that holds more generally. Consider the general multifold systems \eqref{eq:P}, \eqref{eq:D}.  Assume $K$ is as in \eqref{eq:multifold.cone} where each $K_i\subseteq \R^{n_i}, \; i\in I$ is regular.  Furthermore, assume $B_i$ be a compact convex subset of $K_i$ such that
$0\notin B_i$ and $K_i = \cone B_i$ for $i\in I$.  Put
\begin{equation}\label{eq:gral.ellipsoid}
E_i = A_iB_i, \; i\in I.
\end{equation}
Observe that $AK = \cone \co_{i\in I}\{E_i\}$.
Theorem~\ref{thm:main_geometry} can now be stated as follows.

\begin{theorem}\label{thm:partition_detail}  Consider the pair of multifold conic systems \eqref{eq:P}, \eqref{eq:D}.  Assume $K$ is as in \eqref{eq:multifold.cone} and $E_i,\;i\in I$ are as in \eqref{eq:gral.ellipsoid}. Then
\begin{description}
\item[(i)]
The sets $B$, $N$ defined in \eqref{eq:first.conic} satisfy
\[
\begin{array}{c}
B = \{i\in I\,|\,  \ri E_i\cap \Lin(AK) \neq \emptyset \},\\[2ex]
N = \{i\in I\,|\,  E_i \cap \Lin(\overline{AK}) = \emptyset\}.
\end{array}
\]
\item[(ii)]
The sets $B_0$, $N_0$ defined in \eqref{eq:second.conic} satisfy
\[
\begin{array}{c}
B_0 = \{i\in I\,|\,  \ri E_i\cap \Lin(\overline{AK}) \neq \emptyset \},\\[2ex]
N_0 = \{i\in I\,|\,  E_i \cap \Lin(AK) = \emptyset\}.
\end{array}
\]
\end{description}
\end{theorem}

\begin{remark}\label{rem:equiv_cond} {\em
The alternate descriptions for the sets $B,B_0$ in Theorem~\ref{thm:partition_detail} can also be stated as follows.
$$
\begin{array}{rlcl}
& \ri E_i\cap \Lin(AK) \neq \emptyset & \Leftrightarrow & \ri E_i\subseteq \Lin(AK),\\[1ex]
 & \ri E_i \cap \Lin(\overline{AK})  \neq \emptyset
& \Leftrightarrow & \ri E_i \subseteq \Lin(\overline{AK}) .
\end{array}
$$}
\end{remark}

\subsection{Some observations on polyhedral systems}

While for the polyhedral feasibility problem strict complementarity always holds (Proposition~\ref{thm:LP_compl}), one might ask: what happens if each lower-dimensional cone in a
multifold system is itself a product of nonnegative orthants?
Since a linear image of a polyhedral set is closed, from Theorem~\ref{thm:main_partition} it follows that $B=B_0$ and $N=N_0$. Hence, we have only three possible complementarity sets: $B$, $N$ and $C = I\setminus(B\cup N)$.
Any problem with both $B$ and $N$ nonempty could alternatively be considered as a multifold problem with a single cone.  In this case its only index would be in $C$. Therefore, there are polyhedral systems with nonempty $C$.  However, for any polyhedral system the partition sets  $B'$, $N'$ and  $O$ are always empty.

\bibliographystyle{plain}	
\bibliography{references}

\begin{thebibliography}{1}

\bibitem{AlizGold03}
F~Alizadeh and D~Goldfarb.
\newblock Second-order cone programming.
\newblock {\em Math. Progr. Ser. B}, 95(1):3--51, 2003.

\bibitem{CCP_Multifold2008}
D.~Cheung, F.~Cucker, and J.~Pe{\~n}a.
\newblock A condition number for multifold conic systems.
\newblock {\em SIAM J. Optim.}, 19(1):261--280, 2008.

\bibitem{GoldmanTucker}
A.J. Goldman and A.W. Tucker.
\newblock Theory of linear programming.
\newblock {\em Linear Inequalities and Related Systems, H.W. Kuhn and A.W.
  Tucker (eds.), Annals of Mathematical Studies}, 38:53--97, 1956.

\bibitem{JBHULem2001}
J.-B. Hiriart-Urruty and C.~Lemar{\'e}chal.
\newblock {\em Fundamentals of Convex Analysis}.
\newblock Springer-Verlag, Berlin, 2001.

\end{thebibliography}

\end{document}